\documentclass[10pt]{article}
\usepackage[margin=1.1in]{geometry}
\usepackage[utf8]{inputenc}
\usepackage{amsmath}
\usepackage{amssymb}
\usepackage{amsthm}
\usepackage{hyperref}
\usepackage{cleveref}
\usepackage{bm}
\usepackage{tikz}
\usepackage{graphicx}
\usepackage{xcolor}
\usepackage{subcaption}
\usepackage{url}
\usepackage{asymptote}
\usepackage{mathtools}

\usepackage[utf8]{inputenc}
\usepackage{pgfplots}
\pgfplotsset{compat=1.18}

\usepackage[font=small,labelfont=bf]{caption}

\newcommand{\nomatch}{\textcolor{white}{------}}
\newcommand{\blank}{\textcolor{white}{---}}

\newcommand{\ra}{$\rightarrow$}
\newcommand{\la}{$\leftarrow$}

\newtheorem{theorem}{Theorem}

\newtheorem{proposition}[theorem]{Proposition}
\newtheorem{lemma}[theorem]{Lemma}

\usepackage{xcolor}
\newcommand{\Comments}{1}
\newcommand{\mynote}[3]{\ifnum\Comments=1\textcolor{#1}{#2: #3}\fi}

\usepackage[round]{natbib}
\hypersetup{
    colorlinks,
    linkcolor={red!50!black},
    citecolor={blue!50!black},
    urlcolor={blue!80!black}
}

\title{Stable Tables}
\author{Kenny Peng\footnote{I first heard about seating arrangement problems through an episode of \textit{Planet Money}, whose creators I credit for this paper's title \citep{beras2023thanksgiving}. I would especially like to thank Jon Kleinberg for pointing me to the paper of \cite{flory1939intramolecular}, which implies a result for the alternate solution concept of randomized greedy matching (rather than stable matching). I would also like to thank the many people who have entertained my discussion of the present work while seated next to me.}}
\date{}

\begin{document}

\maketitle

\begin{abstract}
    We consider equilibrium one-on-one conversations between neighbors on a circular table, with the goal of assessing the likelihood of a (perhaps) familiar situation: sitting at a table where both of your neighbors are talking to someone else. When $n$ people in a circle randomly prefer their left or right neighbor, we show that the probability a given person is unmatched in equilibrium (i.e., in a stable matching) is
    $$\frac{1}{9} + \left(\frac{1}{2}\right)^n\left(\frac{2n}{3} - \frac{8}{9} + \frac{2}{n}\right)$$
    for odd $n$ and
    $$\frac{1}{9} - \left(\frac{1}{2}\right)^n\left(\frac{2n}{3} - \frac{8}{9}\right)$$
    for even $n$.
    This probability approaches $1/9$ as $n\rightarrow \infty$. We also show that the probability \textit{every} person is matched in equilibrium is $0$ for odd $n$ and $\frac{3^{n/2}-1}{2^{n-1}}$ for even $n$.
\end{abstract}

\section{Introduction}
Perhaps you have sat at a table, or stood around at a party, where both of the people next to you are talking with someone else, leaving you alone to idle awkwardly. To make sense of this phenomenon, we consider a model in which each person at a circular table randomly prefers one of their two neighbors---left with probability $1/2$ and right with probability $1/2$. We then consider the equilibrium set of one-on-one conversations between neighbors (i.e., a stable matching), and determine the probability that a person is unmatched in the resulting equilibrium. (We will show that almost all configurations of preferences admit a unique stable matching.) \Cref{thm1} below provides the likelihood that a person is unmatched for all tables of size $n\ge 1.$

\begin{theorem}\label{thm1}
    The probability $f(n)$ that a given person is unmatched at a table with $n$ people is
    \begin{equation}
        \frac{1}{9} + \left(\frac{1}{2}\right)^n\left(\frac{2n}{3} - \frac{8}{9} + \frac{2}{n}\right)
    \end{equation}
    when $n$ is odd, and
    \begin{equation}
        \frac{1}{9} - \left(\frac{1}{2}\right)^n\left(\frac{2n}{3} - \frac{8}{9}\right)
    \end{equation}
    when $n$ is even.
\end{theorem}
$f(n)$ is formally defined in \Cref{sec:prob-unmatched}. As $n$ grows large, the probability approaches $1/9$ (surprisingly small, perhaps). Results for smaller $n$ are given in \Cref{fig:thm1}.

We provide some high-level intuition. There is a $1/2$ chance that your left neighbor prefers not to talk to you, and a $1/2$ chance that your right neighbor prefers not to talk to you. Therefore, there is a $1/4$ chance neither of your neighbors want to talk to you. But even in this case, not all hope is lost; maybe the person on your left prefers not to talk to you, but the person they want to talk to is busy talking to someone else. So---luckily for you---they turn back in your direction. It turns out that the probability that your neighbor both prefers not to talk to you \textit{and has a better option} is about $1/3,$ meaning that you only have a roughly $1/9$ chance of being unmatched. 

\begin{figure}
\begin{tikzpicture}
\begin{axis}[title=\Cref{thm1}: Probability $f(n)$ of being unmatched at a table of $n$ people, xlabel=$n$ (\# of people),
  ylabel=probability alone, width=\linewidth,height=7cm, xmin=0.7, xmax=12.3, ymin=-0.025, ymax=1.15]
  \addplot[mark=none, purple, thick, dashed] coordinates {(0.7,0.111) (12.3,0.111)};
  \addlegendentry{$\frac{1}{9}$ probability in limit}
  \addplot[mark=none, lightgray] coordinates {(0.7,0) (12.3,0)};
  \addplot[mark=none, lightgray] coordinates {(0.7,0.2) (12.3,0.2)};
  \addplot[mark=none, lightgray] coordinates {(0.7,0.4) (12.3,0.4)};
  \addplot[mark=none, lightgray] coordinates {(0.7,0.4) (12.3,0.4)};
  \addplot[mark=none, lightgray] coordinates {(0.7,0.6) (12.3,0.6)};
  \addplot[mark=none, lightgray] coordinates {(0.7,0.8) (12.3,0.8)};
  \addplot[mark=none, lightgray] coordinates {(0.7,1) (12.3,1)};
  \addplot[mark=none, black] coordinates {(0.7,0) (12.3,0)};
    \addplot[
        scatter/classes={a={red}, b={blue}},
        scatter, mark=*, only marks, 
        scatter src=explicit symbolic,
        nodes near coords*={\Label},
        visualization depends on={value \thisrow{label} \as \Label} %
    ] table [meta=class] {
        x y class label
        1 1 a $1$
        2 0 b $0$
        3 0.333 a $\frac{1}{3}$
        4 0 b $0$
        5 0.2 a $\frac{1}{5}$
        6 0.0625 b $\frac{1}{16}$
        7 0.1428 a $\frac{1}{7}$
        8 0.09375 b $\frac{3}{32}$
        9 0.1215 a $\frac{35}{288}$
        10 0.1055 b $\frac{27}{256}$
        11 0.1143 a $\frac{161}{1408}$
        12 0.1094 b $\frac{7}{64}$ 
    };
\end{axis}
\end{tikzpicture}
\caption{The probability of being alone at a table of $n$ people, for $1\le n\le 12.$ Red points mark odd-sized tables, and blue points mark even-sized tables. Notice that the probability quickly approaches $\frac{1}{9}$ as $n$ increases, though for small $n$ the probability can deviate significantly from the asymptotic value. The probability is always above $\frac{1}{9}$ for odd $n$ and below $\frac{1}{9}$ for even $n$, a fact intuitively explained by the necessity of someone being unmatched when there is an odd number of people.}
\label{fig:thm1}
\end{figure}
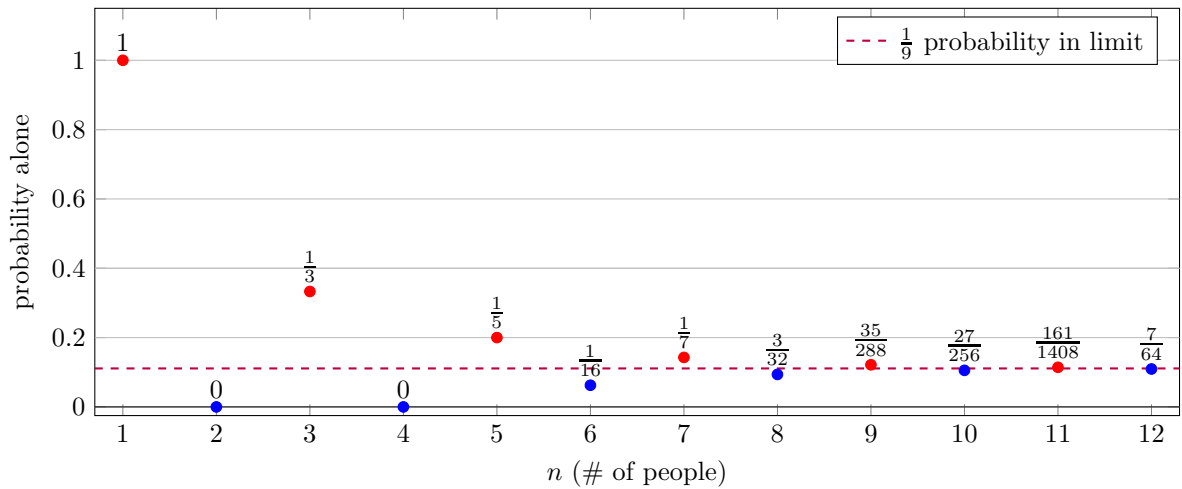

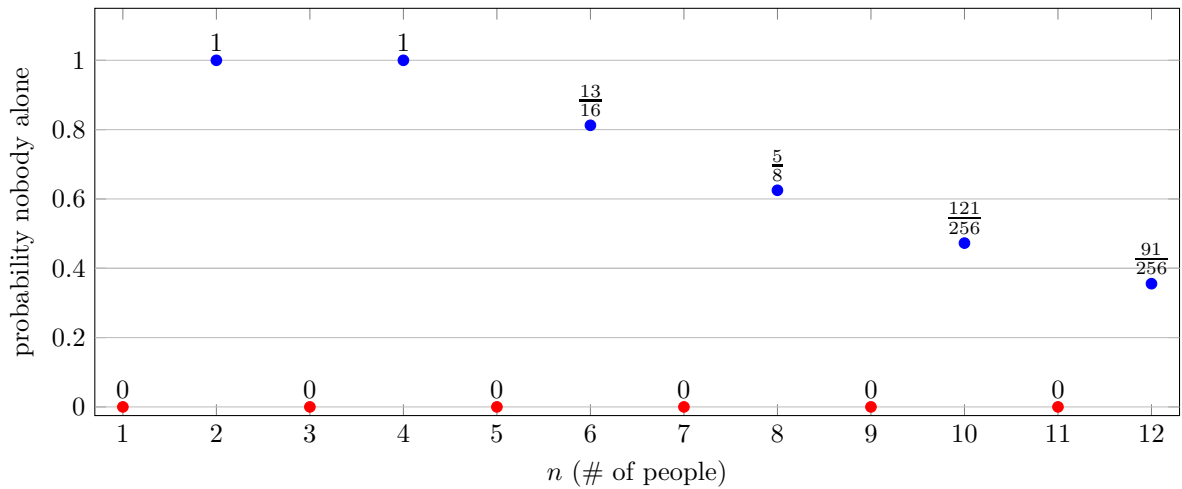
\begin{figure}
\begin{tikzpicture}
\begin{axis}[title=\Cref{thm2}: Probability $g(n)$ that nobody is unmatched at a table of $n$ people, xlabel=$n$ (\# of people),
  ylabel=probability nobody alone, width=\linewidth,height=7cm, xmin=0.7, xmax=12.3, ymin=-0.025, ymax=1.15]
  \addplot[mark=none, lightgray] coordinates {(0.7,0) (12.3,0)};
  \addplot[mark=none, lightgray] coordinates {(0.7,0.2) (12.3,0.2)};
  \addplot[mark=none, lightgray] coordinates {(0.7,0.4) (12.3,0.4)};
  \addplot[mark=none, lightgray] coordinates {(0.7,0.4) (12.3,0.4)};
  \addplot[mark=none, lightgray] coordinates {(0.7,0.6) (12.3,0.6)};
  \addplot[mark=none, lightgray] coordinates {(0.7,0.8) (12.3,0.8)};
  \addplot[mark=none, lightgray] coordinates {(0.7,1) (12.3,1)};
    \addplot[
        scatter/classes={a={red}, b={blue}},
        scatter, mark=*, only marks, 
        scatter src=explicit symbolic,
        nodes near coords*={\Label},
        visualization depends on={value \thisrow{label} \as \Label} %
    ] table [meta=class] {
        x y class label
        1 0 a $0$
        2 1 b $1$
        3 0 a $0$
        4 1 b $1$
        5 0 a $0$
        6 0.8125 b $\frac{13}{16}$
        7 0 a $0$
        8 0.625 b $\frac{5}{8}$
        9 0 a $0$
        10 0.4726 b $\frac{121}{256}$
        11 0 a $0$
        12 0.3555 b $\frac{91}{256}$ 
    };
\end{axis}
\end{tikzpicture}

\caption{The probability that everybody is matched at a table of $n$ people, for $1\le n\le 12.$ Red points mark odd tables, and blue points even tables. Clearly, the probability is only non-zero for even $n$. While the probability decays exponentially, for small even $n$, the probability is reasonable.}
\label{fig:thm2}

\end{figure}

\paragraph{}
We then consider the probability that \textit{no} person at a table is left alone in a stable matching. Clearly, this is only possible when there an even number of people.

\begin{theorem}\label{thm2}
    The probability $g(n)$ that every person is matched at a table with an even number $n$ people is
    \begin{equation}
        \frac{3^{n/2}-1}{2^{n-1}}.
    \end{equation}
\end{theorem}
$g(n)$ is formally defined in \Cref{sec:prob-no-unmatched}. This probability decays exponentially fast---it is $O(0.867^{n})$---but is reasonably large for small $n$. For example, it is greater than $\frac{1}{2}$ for tables of size $2, 4, 6,$ and $8$ (see \Cref{fig:thm2}).

Showing \Cref{thm2} relies on the following characterization of when every person is matched. Define a \textit{natural pair} to be any two neighbors who mutually prefer each other. Then, we show that everyone is matched at a table if and only if there is an even number of people seated between any two consecutive natural pairs. (The only-if direction is clear: in equilibrium, natural pairs must be matched to each other, so if there are an odd number of people between two natural pairs, then one person in this stretch must be left unmatched.) In this way, \Cref{thm2} reduces to counting the number of preferences that correspond to such a configuration.

The paper proceeds as follows. \Cref{sec:related-work} describes related work in seat arrangement problems, matching markets with random preferences, and random matching algorithms. \Cref{sec:model} introduces the stable matching model. In \Cref{sec:analysis}, we prove several basic results about stable matching on a circle---in particular, demonstrating that there is almost always a unique stable matching, and characterizing when a given person is unmatched. In \Cref{sec:prob-unmatched}, we formalize and prove \Cref{thm1}. In \Cref{sec:prob-no-unmatched}, we formalize and prove \Cref{thm2}. Then, in \Cref{sec:random-matching}, we briefly compare these results with those arising from an alternate solution concept, randomized greedy matching \citep{dyer1991randomized}, in which matches are sequentially chosen uniformly at random. In particular, under randomized greedy matching, a result shown by \cite{flory1939intramolecular} implies the probability an individual on a circular table is unmatched approaches $1/e^2>1/9$ for large $n$. In \Cref{sec:conclusion}, we briefly conclude.

\section{Related work}\label{sec:related-work} 

Problems related to seating arrangements and stability have received recent attention. In particular, a number of papers study \textit{hedonic seat arrangement problems} in which a person's utility is given by the sum of their cardinal utilities over their neighbors \citep{bodlaender2020hedonic, ceylan2023optimal, berriaud2023stable}. One interesting question in this line of work considers when a swap-stable seating arrangement exists---one in which no two people prefer to swap positions---and how to compute such an arrangement. While this line of work considers stability with respect to where people are seated (i.e., the seating arrangement), we focus on stability with respect to which conversations form. In particular, we do not allow individuals to change seats, but rather allow individuals to choose which neighbor to converse with. (We also consider ordinal rather than cardinal preferences.) While we do not focus on how to choose a seating arrangement, one may pose an related question in the model we consider: Given the preferences of $n$ people, can one always seat them at a circular table such that every person ends up in a conversation in equilibrium? When $n$ is even and all people have full preference lists, the answer is yes, and can be shown with a short argument using standard stable matching theory.\footnote{Split the $n=2m$ people arbitrarily into two groups of $m$, compute a stable matching between these two groups (which must exist by \cite{gale1962college}). This matching must be a perfect matching since every person has a full preference list. Then seat the resulting matches (pairs of people) along the table such that people in the first group occupy even positions and people in the second group occupy odd positions. Then, in equilibrium, each person converses with their match, as any blocking pair along the table would contradict the stability of the original matching.}

Mathematically, the present work is more directly related to the substantial literature studying random matching markets---i.e., those in which participants have random preferences. For example, \cite{pittel1989average} studies the expected number of stable matchings in a one-to-one two-sided matching market with $n$ participants on each side, each of whom have uniformly random preferences; Pittel shows that there are approximately $e^{-1}n\log n$ stable matchings on average. \cite{ashlagi2017unbalanced} consider the same setup, but with $n$ participants on one side and $n+1$ on the other, finding instead a vanishing number of stable matchings. \cite{immorlica2003marriage} similarly demonstrate that there is a vanishing number of stable matchings when participants have short lists. As a consequence of our analysis (\Cref{prop:stable}), we contribute a basic result in this direction, providing a network topology (the cycle graph), under which almost all preference configurations result in a unique stable matching.

Like us, \cite{arnosti2023lottery} focuses on unmatched participants in random matching markets, comparing the number of unmatched students in many-to-one stable matching under different tie-breaking mechanisms. (This is a key statistic in school choice policy discussions \citep{abdulkadirouglu2009strategy,de2023performance}.) Unmatched students arise in Arnosti's setting due to the existence of incomplete preference lists. Our setting can be viewed as an extreme case in which participants all have preference lists of length exactly two (over their two neighbors). We note that our setting is an instance of standard one-to-one two-sided matching exactly when the size of the table is even, since even-sized cyclic graphs are bipartite.

Our work is also related to random matching algorithms on general graphs, in which matches are sequentially chosen at random. \cite{dyer1991randomized} provides bounds on the ratio between the average performance of ``randomized greedy matching'' in comparison to the maximum matching, including for restricted sets of graphs (like trees). \cite{dyer1993average} considers the same question, but for random sparse graphs. Randomized greedy matching can be viewed as an alternate solution concept under which to study our central problem: the probability that an individual is left conversation-less at a dinner table. A result of \cite{flory1939intramolecular} implies that randomized greedy matching matches approximately a $1-1/e^2$ proportion of individuals arranged along a large cycle graph, while \Cref{thm1} of the present work implies that a stable matching arising from random preferences matches a $8/9 > 1 - 1/e^2$ proportion of individuals. It would be interesting to compare the performance of stable matching with random preferences and greedy random matching on more general graphs. (Of course, any such analysis must deal with the possibility that there does not exist a stable match.)

The present work should not be confused with papers that consider the \textit{physical} stability (i.e., wobbliness) of four-legged tables \citep{martin2007stability, baritompa2005mathematical}. A remarkable set of results shows that as long as the local slope of the ground is not too significant (in fact, $<35.26$ degrees), a four-legged table with equal-length legs can always be rotated to a stable position.

\section{Model}\label{sec:model}

We introduce the model in three parts. First, we formalize a set of one-on-one conversations around a table as a \textit{matching}. Second, we formalize the \textit{preferences} of each person over their two neighbors. Finally, we introduce \textit{stability} \citep{gale1962college} as the solution concept from which to determine the matching that forms given a set of preferences. This provides a setup for analyzing what conversations arise given randomly instantiated preferences.

\paragraph{Matching around a table.} Consider $n$ people seated around a circular table, labeled $0, 1, 2, \cdots, n-1$. Then each person $i$ has two neighbors, $i-1$ (their left neighbor) and $i+1$ (their right neighbor), where arithmetic is modulo $n$. For example, in a six-person table, person $3$ has neighbors $2$ and $4$, and person $5$ has neighbors $4$ and $0$.

We model one-on-one conversations between neighbors around the table: each person converses with at most one other person, and they can only converse with their two neighbors. Formally, a \textit{matching} is a function $\mu:\{0,1,\cdots,n-1\}\rightarrow \{0,1,\cdots,n-1\},$ where $\mu(i)$ is to be interpreted as the person whom $i$ converses with. The function $\mu$ must therefore satisfy two properties:
\begin{enumerate}
    \item $\mu(i)\in \{i-1, i, i+1\}.$ In other words, a person is either matched to one of their two neighbors, or, if $\mu(i)=i$, they are unmatched.
    \item $\mu(i)=j$ implies $\mu(j)=i.$ In other words, conversations must be mutual.
\end{enumerate}
If $\mu(i)=i,$ we say that person $i$ is \textit{unmatched} in $\mu$. A primary goal of this paper is to analyze the likelihood that a person is unmatched. To answer this question, we must consider what matching forms at a table with $n$ people. We model this outcome as arising from the preferences of each person. 

\paragraph{Preferences over neighbors.}
We assume that each person strictly prefers their left or right neighbor. (And prefers both neighbors to being alone.)
The preferences of $n$ people around a table can therefore be represented as a string $\pi \in \Pi_n := \{L,R\}^n$, where the $i$-th character of the $n$-character string (which we index starting at $0$) is the preference of person $i$. For example, the string
\begin{equation}
    \pi = RRLR
\end{equation}
corresponds to a four-person table in which person $0$ prefers $1$, $1$ prefers $2$, $2$ prefers $1$, and $3$ prefers $0$. We may also illustrate this diagrammatically as
\begin{equation}
    \text{0\ra\blank 1\ra\la 2\nomatch 3\ra},
\end{equation}
where arrows indicate the direction of each person's preference.

We define a \textit{natural pair} to be two neighbors who mutually prefer each other. In the example above, 1 and 2 form a natural pair. Notice that a natural pair consists of an occurrence of the substring $RL$ in $\pi$, where this substring can ``wrap around'' the end of a string. For example, for $\pi=LLRR$, person $3$ and person $0$ form a natural pair. Observe also that $\pi$ contains a natural pair if and only if it contains at least one $R$ and at least one $L$. We call such $\pi$ \textit{regular}. Meanwhile, we say $\pi$ is \textit{irregular} if and only if it is not regular; in other words, $\pi$ is irregular if and only if it contains only $R$'s or only $L$'s.

We also say that a person is \textit{preferred} if at least one of their neighbors prefers them.

\paragraph{Stable matching.} Given preferences $\pi$, a matching $\mu$ is \textit{stable} if and only if there does not exist a \textit{blocking pair}, two neighbors $i$ and $i+1$ that:
\begin{enumerate}
    \item are not currently matched to each other, i.e., $\mu(i)\neq i+1$, and
    \item both prefer each other over their current match (or lack of match).
\end{enumerate}
In other words, a matching (i.e., set of one-on-one conversations) is stable if and only if no two neighbors would defect to converse with each other. A stable matching thus reflects an \textit{equilibrium} set of conversations. Returning to the example $\pi=RRLR$, the matching where $\mu(0)=1$ and $\mu(2)=3$ is not stable since $1$ and $2$ would defect to converse with each other. Indeed, this example conveys a basic insight: that natural pairs must be matched together in stable matchings.

\section{Analysis of stable matchings}\label{sec:analysis}
In this section, we show some basic results pertaining to stable matchings on a table. The key results in this section establish when stable matchings exist and are unique (\Cref{prop:stable}), as well as when a person is unmatched and when no person is unmatched (\Cref{lem:unmatched-conditions}). I begin by characterizing when a matching is stable with respect to a given set of preferences:

\begin{lemma}\label{lem:stable-conditions}
A matching $\mu$ is stable under preferences $\pi\in \Pi_n$ if and only if:
\begin{enumerate}
    \item all natural pairs are matched to each other,
    \item all preferred people are matched to someone,
    \item and no two neighbors are both unmatched.
\end{enumerate}
\end{lemma}

\begin{proof}
We begin with the only if direction. We show that $\mu$ is violated whenever any of the conditions are violated. Suppose condition $1$ were violated, so that there is a natural pair that is not matched to each other. Then $\mu$ is unstable since these two neighbors would defect to match with each other. Now suppose condition $2$ were violated, so that there is a preferred person $i$ who is not matched. Then one of their neighbors $j\in \{i-1,i+1\}$ prefers $i$. Then $i$ and $j$ would defect to be matched with each other. Finally, suppose condition $3$ were violated, so that there are two neighbors are both unmatched. Then these two neighbors would defect to be matched with each other.

We now show the if direction. We show that if $\mu$ satisfies these three conditions, then there do not exist two neighbors $i$ and $j$ that are not currently matched and who would defect to be matched with each other. We consider three cases: when $i$ and $j$ both prefer each other, when $i$ prefers $j$ but $j$ does not prefer $i$, and when neither $i$ and $j$ prefer each other. The first case is not possible when $\mu$ satisfies condition $1$. In the second case, since $j$ is preferred, $j$ must be matched by condition $2$. Since they are not matched to $i$, they are matched to their other neighbor, which is their preferred neighbor. Therefore, $j$ would not want to defect from their match to match with $i$. In the third case, neither $i$ or $j$ prefer each other. By condition $3$, at least one of them is matched. This person is matched to their preferred neighbor, so they would not defect. Therefore, any pair of neighbors who are not currently matched would not defect to be matched to each other, showing that $\mu$ must be stable.
\end{proof}

\begin{proposition}\label{prop:stable}
    For preferences $\pi\in \Pi_n$:
    \begin{enumerate}
        \item[(i)] If $\pi$ is regular, there is a unique stable matching.
        \item[(ii)] If $\pi$ is irregular and $n$ is even, a matching is stable if and only if no person is unmatched.
        \item[(iii)] If $\pi$ is irregular and $n$ is odd, there exists no stable matching.\footnote{Notice that since a cycle graph with an odd number of vertices is not bipartite, a stable matching is not guaranteed from Gale-Shapley.}
    \end{enumerate}
\end{proposition}

\begin{proof}
Part (i), which is more involved than the other cases, is a corollary of the subsequent lemma (\Cref{lem:unmatched-conditions}), in which we show something more specific. Therefore, we focus on showing (ii) and (iii).

To show part (ii), consider $\pi\in \Pi_n$ such that $\pi$ is irregular and $n$ is even. Since $\pi$ is irregular, there are no natural pairs. Therefore, \Cref{lem:stable-conditions} implies that a matching $\mu$ is stable under $\pi$ if and only if all preferred people are matched to someone and no two neighbors are both unmatched. Since every person is preferred when $\pi$ is irregular, we see that a matching is stable if and only if no person is unmatched.

To show part (iii), consider $\pi\in \Pi_n$ such that $\pi$ is irregular and $n$ is odd. Again, since $\pi$ is irregular, there are no natural pairs, so (by \Cref{lem:stable-conditions}) $\mu$ is stable under $\pi$ if and only if all preferred people are matched and no two neighbors are both unmatched. Since there are an odd number of people, it is impossible for every person to be matched. It follows that no stable matching exists.
\end{proof}

\Cref{prop:stable} may be considered interesting on it's own, as it implies that there is almost always a unique stable matching. In general, the number of stable matchings can be large. \cite{pittel1989average} showed that in a one-to-one two-sided matching market with $n$ participants on each side, each with uniformly random preferences, that the expected number of stable matchings is roughly $e^{-1}n\log n.$ By restricting preference lists according to the cycle graph structure, \Cref{prop:stable} implies that the expected number of stable matchings is approximately $1$. This is in agreement with a main result of \cite{immorlica2003marriage}, which, while not implying \Cref{prop:stable}, shows that matching markets with restricted preference lists on one side admit a small number of stable matchings.

\paragraph{} To complete the proof of \Cref{prop:stable}, we turn to the case when $\pi$ is regular, which is more involved, and will form the bulk of our analysis. The following definitions will be useful. For a person $i$ and preferences $\pi$, define $s_\pi(i)$ to be the number of seats to the left of $i$ until someone prefers their right neighbor. Similarly, define $t_\pi(i)$ to be the number of seats to the right of $i$ until someone prefers their left neighbor. We will generally drop the subscript $\pi$ when the preferences are clear. Note that for a preferred person, either $s(i)=1$ or $t(i)=1$, since if their left neighbor prefers them, then $s(i)=1$ and if their right neighbor prefers them then $t(i)=1$.

We may now show that when $\pi$ is regular, there exists a unique stable matching, showing \Cref{prop:stable}(i). The following lemma further shows two key properties of this unique stable matching, which we will use to show our main results, Theorems 1 and 2.

\begin{lemma}\label{lem:unmatched-conditions}
For regular preferences $\pi\in \Pi_n$, there is a unique stable matching $\mu_\pi$, and
\begin{enumerate}
    \item[(i)] person $i$ is unmatched in $\mu_\pi$ if and only if both $s(i)$ and $t(i)$ are even,
    \item[(ii)] no person is unmatched in $\mu_\pi$ if and only if there is an even number of seats between any two adjacent natural pairs.
\end{enumerate}
\end{lemma}

\begin{proof}
We will use \Cref{lem:stable-conditions} to exactly characterize the unique stable matching, and show (i) and (ii) in the process. By \Cref{lem:stable-conditions}, we know that neighbors that form a natural pair must be matched in any stable matching. With this in mind, we now determine what happens between two adjacent natural pairs. Consider a substring of $\pi'$ of $\pi$ that consists of the people in between two adjacent natural pairs. (If there is only one natural pair, then there is exactly one such substring, consisting of those people not in the natural pair.) For example, one such substring is the underlined portion of
\begin{equation}
LRL\underline{LLR}RLL.
\end{equation}
Notice that any such substring $\pi'$ must have the following format: a (possibly empty) string of $L$'s, followed by a (possibly empty) string of $R$'s. Now suppose $\pi'$ consist of the people strictly between person $a$ and person $b$. We then break up the people in $\pi'$ between the two natural pairs into three zones:
\begin{itemize}
     \item Zone 1 begins on the left side and contains the maximum even number of $L$'s following the starting natural pair ($RL$). So zone $1$ contains the people $a+1, a+2, \cdots, a+2p$ for some $p\ge 0$ even.
     \item Zone $3$ begins on the right side and contains the maximum even number of $R$'s preceding the ending natural pair ($RL$). So zone $3$ contains the people $b-1, b-2, \cdots, b-2q$ for some $q\ge 0$ even.
     \item Zone $2$ contains the remaining people in between zones $1$ and $3$. Zone $2$ thus consists of $LR$, $L$, $R$, or is empty. (If there were at least two $L$'s or $R$'s in zone $2$, then they would be incorporated into zone $1$ or $3$.)
\end{itemize}
In the example above, zone $1$ consists of $LL$, zone $3$ consists of $R$, and zone $2$ is empty. We now show that $\mu$ is stable if and only if for all $\pi'$,
\begin{itemize}
    \item $\mu(a+k)=a+k+1$ for all $i<p$ odd (so everyone in zone $1$ is matched to a neighbor also in zone $1$);
    \item $\mu(b-k)=b-k-1$ for all $i<q$ odd (so everyone in zone $3$ is matched to a neighbor also in zone $3$);
    \item $\mu(i)=i+1$ if zone $2$ contains two people, $i$ and $i+1$;
    \item and $\mu(i)=i$ if zone $2$ contains only one person, $i$.
\end{itemize}
Notice that this statement implies a unique stable matching, since it specifies where each person must be matched. It is easy to see the if direction, by observing that the matching satisfies the conditions in \Cref{lem:stable-conditions}. We now show the only if direction. Suppose that a matching $\mu$ satisfies the conditions in \Cref{lem:stable-conditions}. Then observe that $a+1$ and $a+2$ must be matched in any stable matching since $a+1$ must be matched (by condition $2$, because they are preferred by $a+2$), and cannot be matched with person $a$ since $(a, a-1)$ is a natural pair. It follows that $a+3$ must be matched with $a+4$ and so forth. Similarly, $b-1$ and $b-2$ must be matched, as well as $b-3$ and $b-4$ and so forth. This shows that each person in zone $1$ and zone $3$ is matched to another person in their respective zone. It remains to consider zone $2$. If zone $2$ consists of two people, $i$ and $i+1$, then at least one must be matched by condition $3$. Suppose that $i$ is matched. Then $i$ must be matched to $i+1$ since $i-1=a+2p$ is in zone $1$ and is thus matched to $a+2p-1.$ Similarly, we can show that if $i+1$ is matched, then $i+1$ must be matched to $i$. Therefore, if zone $2$ contains two neighbors, then they must be matched to each other. Finally, if zone $2$ contains just one person, then that person is unmatched since both of their neighbors ($a+2p$ and $b-2q$) are matched to others.

Taking account, we observe that no person is unmatched if and only if zone $2$ is empty or contains two people. Since zones $1$ and $3$ contain an even number of people by definition, this implies part (ii), that no person is unmatched in the unique stable matching if and only if there is an even number of people between any two adjacent natural pairs.

To show part (i), notice that every person $i$ in zone $1$, zone $3$, or who is part of a natural pair is matched. Of these people, every person is certainly preferred except perhaps the right-most person in zone $1$ and the left-most person in zone $3$. If $i$ is preferred, then $s(i)=1$ or $t(i)=1$. If $i$ is the right-most person in zone $1$, then they are by definition an odd number of people away from the first $R$ on their left (noting that there is an even number of people in zone $1$), so $s(i)$ is odd. Similarly, if $i$ is the left-most person in zone $3$, then $i$ is an odd number of people away from the first $L$ on their right, so $t(i)$ is odd. So for people who are in zone $1$, zone $3$, or who are in natural pairs, either $s(i)$ or $t(i)$ is odd. Now consider the people in zone $2$. Consider when zone $2$ contains two people $i$ and $i+1$, represented by $LR$. Both are matched in this case. Now observe that the person $i$ is an odd number of people from the first $L$ on their right, since there is one $R$ to their right in zone $2$ followed by an even number of $R$'s in zone $3$. So $t(i)$ is odd. Similarly, person $i+1$ is an odd number of people from the first $R$ on their left, so $s(i+1)$ is odd. When zone $2$ only consists of one person, person $i$, then that person is unmatched. In this case, both $s(i)$ and $t(i)$ are even: on their left, there is an even number of $L$'s in zone $1$ followed by an $L$ and then an $R$, and on their right, there is an even number of $R$'s in zone $3$ followed by an $R$ and then an $L$. Thus, we have shown that person $i$ is unmatched if and only if $s(i)$ and $t(i)$ are both even.
\end{proof}

\section{Probability of being unmatched}\label{sec:prob-unmatched}
We may now consider the likelihood that person $i$ is unmatched at a table where people have random preferences---i.e., where each person uniformly at random prefers either their left or right neighbor. In other words, for $\pi$ drawn uniformly at random from $\Pi_n$, what is the probability that $i$ is unmatched? To answer this, we will use the stable matching theory developed above in \Cref{prop:stable}. If $\pi$ is regular, there is a unique stable matching $\mu_\pi$, so it is natural to suppose that $i$ is unmatched in this case if and only if $\mu_\pi(i)=i.$ If $\pi$ is irregular and $n$ is even, then both possible stable matchings are perfect (everybody is matched), so we suppose $i$ is always matched in this case. Finally, if $\pi$ is irregular and $n$ is odd, then there does not exist any stable matching, so we separately specify what happens; for simplicity, we suppose that in this case, $i$ is unmatched with probability $\frac{1}{n}$ (that there is one person randomly left out). In this way, by considering preferences $\pi\in \Pi_n$ uniformly at random, the probability $f(n)$ that person $i$ is unmatched at a table with $n$ people is well-defined.

We may now show \Cref{thm1}, exactly computing $f(n)$ for $n\ge 1$.

\paragraph{Proof of \Cref{thm1}.}
We first consider the case when $n=2m$ is even, computing the probability that $i$ is unmatched. \Cref{prop:stable}(ii) implies that $i$ is never unmatched when $\pi$ is irregular and $n$ is even. Therefore, we determine the probability that $i$ is unmatched and $\pi$ is regular. Recall from \Cref{lem:unmatched-conditions}(i) that if $\pi$ is regular, then $i$ is unmatched if and only if $s_\pi(i)$ and $t_\pi(i)$ are both even, where we recall that $s_\pi(i)$ is the number of seats to the left of $i$ until the first occurrence of an $R$, and $t_\pi(i)$ is the number of seats to the right of $i$ until the first occurrence of an $L$.

We compute the likelihood of regular preferences $\pi$ such that $s_\pi(i)$ and $t_\pi(i)$ are both even. To do this, we first consider the number of preferences $\pi$ such that $s_\pi(i)=2c$ and $t_\pi(i)=2d.$ There exists such $\pi$ if and only if $c,d\ge 1$ and $c+d\le m-1.$ Taking such $c$ and $d$, the probability that a randomly drawn $\pi\in \Pi_n$ satisfies $s_\pi(i)=2c$ and $t_\pi(i)=2d$ is
\begin{equation}
    \left(\frac{1}{2}\right)^{2c} \left(\frac{1}{2}\right)^{2d},
\end{equation}
where the probability comes from observing that the $c-1$ characters to the left of $i$ must all be $L$'s and the $c$-th character to the left of $i$ must be $R$, the $d-1$ characters to the right of $i$ must all be $R$'s and the $c$-th character to the right of $i$ must be $L$. So the total probability that $\pi$ is regular and $i$ is unmatched is
\begin{align}
    \sum_{\substack{c,d\ge 1 \\ c+d\le m-1}} \left(\frac{1}{2}\right)^{2c} \left(\frac{1}{2}\right)^{2d} &= \sum_{c=1}^{m-2} \left(\frac{1}{2}\right)^{2c} \sum_{d=1}^{m-1-c} \left(\frac{1}{2}\right)^{2d}\\
    &= \frac{1}{9} - \left(\frac{1}{4}\right)^{m}\left(\frac{4m}{3} - \frac{8}{9}\right),\label{eq:orange}
\end{align}
where the calculation of the summation is shown in the appendix (\Cref{prop:thm1-even}). Substituting $m=n/2$, we have that
\begin{equation}
    f(n) = \frac{1}{9} - \left(\frac{1}{2}\right)^n\left(\frac{2n}{3} - \frac{8}{9}\right),
\end{equation} 
as desired.

We leave the case of odd $n$ to the appendix. The proof is very similar, though a couple additional cases need to be considered separately.

\section{Probability everybody is matched}\label{sec:prob-no-unmatched}
We now consider the probability that no person is unmatched when preferences $\pi\in \Pi_n$ are chosen uniformly at random. First observe that when $n$ is odd, this is impossible, so the probability is $0$. So we focus on when $n$ is even. In this case, from \Cref{prop:stable}, if $\pi$ is regular, then there exists a unique stable matching, so we suppose that every person is matched if and only if this unique stable matching is perfect. Moreover, when $\pi$ is irregular, both stable matchings are perfect, so we assume that everyone is matched in this case. Therefore, the probability $g(n)$ that every person is matched when preferences are chosen uniformly at random is well defined.

\paragraph{Proof of \Cref{thm2}.}
We count the total number of preferences that result in a perfect stable matching $\pi$. The two irregular preferences (all $R$'s and all $L$'s) both result in perfect stable matchings.

Recall from \Cref{lem:unmatched-conditions}(ii) that regular preferences $\pi$ induce a stable matching that is perfect if and only if there is an even number of people between any two adjacent natural pairs. We say that such a string $\pi$ is \textit{even-spaced}. For example, the $10$-character strings
\begin{align*}
    & R\underline{RL} LL\underline{RL}RRR \\
    & \underline{L}\underline{RL} LR\underline{RL} LR\underline{R}
\end{align*}
are both even-spaced, and thus both correspond to perfect stable matchings. (We have underlined natural pairs $RL$.) Notice that in the second string, there is an occurrence of $RL$ that wraps from the end of the string to the beginning.

Now we count the number of even-spaced regular strings $\pi$. A regular string is even-spaced if and only if either (1) all occurrences of $RL$ are in positions $(i, i+1)$ for $i$ odd, or (2) all occurrences of $RL$ are in positions $(i, i+1)$ for $i$ even. To count the number of regular strings satisfying condition (1), consider the choices of the characters occupying the pairs $(0,1), (2,3), \cdots, (n-2,n-1).$ These can each by occupied by $RR, LR,$ and $LL$, but not by $RL$. Therefore, there are $3^{n/2}-2$ regular strings $\pi$ satisfying condition (1), where we accounted for the $2$ irregular strings that also satisfy the condition. Similarly, there are $3^{n/2}-2$ regular strings $\pi$ that satisfy condition (2). Note that there are no regular strings that satisfy both conditions (since there must exist an occurrence of $RL$ that either starts in an odd position or even position).

In total, there are $2 + (3^{n/2}-2) + (3^{n/2}-2) = 2\cdot 3^{n/2} - 2$ preferences that result in perfect stable matchings. Therefore, the probability $g(n)$ that randomly chosen preferences $\pi$ result in a perfect stable matching is
\begin{equation}
    \frac{2\cdot 3^{n/2} - 2}{2^n} = \frac{3^{n/2}-1}{2^{n-1}},
\end{equation}
as desired.

\section{Comparison to random matching}\label{sec:random-matching}

One might imagine an alternate way in which conversations are determined at a table, where conversations randomly arise and are held fixed, rather than settling into equilibrium. (Even if you would prefer to turn away to speak with your neighbor, you might be stuck in your current conversation due to social conventions.) It is natural to consider differences between this ``random matching'' process and stable matching.

We provide the analogue of \Cref{thm1} in this setting, giving the probability that a given individual is unmatched under the ``randomized greedy matching'' procedure studied by \cite{dyer1991randomized}, in which matches are sequentially chosen uniformly at random from the set of remaining possible matches.

For example, consider a table with six people. Then, the procedure first selects some pair of neighbors to match. This leaves four remaining people, arranged in a line. There are three remaining possible matches. If the ``middle'' match is chosen, then this results in two unmatched people. If either of the ``outside'' matches are chosen, then there remains two more people, who are neighbors and thus also match. (This implies a $2/3$ chance that no person is unmatched in a table of six.)

Interestingly, the following result, which gives the probability an individual is unmatched under randomized greedy matching, follows directly from a result of \cite{flory1939intramolecular} motivated by particular reactions involving polymers (see page 2).\footnote{Flory won the 1974 Nobel Prize in Chemistry.} (Flory was interested in chemical reactions that required adjacent subunits. During a reaction involving a polymer with long chain length, some subunits would become isolated, and therefore not complete the reaction.)

\begin{proposition}[Corollary of \cite{flory1939intramolecular}]\label{prop:random1}
    The probability $f^*(n)$ that a given person is unmatched in a table with $n$ people under random matching is
    \begin{align}
        f^*(n) &= \sum_{k=0}^{n-3} \frac{(-2)^{k}(n-2-k)}{k!}.
    \end{align}
\end{proposition}
In particular, the result implies that
\begin{equation}
    \lim_{n\rightarrow\infty} f^*(n) = \frac{1}{e^2},
\end{equation}
which is slightly higher than the limiting probability that an individual is unmatched under stable matching,
$\lim_{n\rightarrow \infty} f(n) = 1/9.$ In other words, randomized greedy matching results in more unmatched individuals on average. Flory's result directly corresponds to a setting in which $n$ people are seated along a line (rather than a circle). \Cref{prop:random1} follows after observing that after a single match is chosen on a circle of size $n$, the problem reduces to a line of size $n-2$.

\section{Conclusion}\label{sec:conclusion}
We considered stable matching on a cycle graph with uniformly random preferences, with the goal of better understanding a particular circumstance: being seated at a table in which both of your neighbors are speaking to someone else. For tables of size $n\ge 1$, we determined exact probabilities for (1) the probability that an individual is unmatched, and (2) the probability that no individual is unmatched. In particular, \Cref{thm1} implies that the probability that an individual is unmatched approaches $1/9$ when $n$ is large. As our result supposes that preferences are selected uniformly at random, one might compare the estimate of $1/9$ with personal experience to assess the innocuousness of this assumption (and the strength of one's social skills). It would be interesting to generalize the current setup to other graph structures or preference distributions. It would also be interesting to consider solution concepts that allows for conversations with more than two participants.

{
\bibliographystyle{alpha}
\bibliography{bib}
}

\appendix

\section{Proof of Theorem 1 (Supplement)}

We first state and prove the identify in \Cref{eq:orange}.

\begin{proposition}\label{prop:thm1-even}
\begin{equation}
    \sum_{c=1}^{m-2} \left(\frac{1}{2}\right)^{2c} \sum_{d=1}^{m-1-c} \left(\frac{1}{2}\right)^{2d}
    = \frac{1}{9} - \left(\frac{1}{4}\right)^{m}\left(\frac{4m}{3} - \frac{8}{9}\right).
\end{equation}
\end{proposition}

\begin{proof}
First observe that
    \begin{align}
        \sum_{d=1}^{m-1-c} \left(\frac{1}{2}\right)^{2d}
        = \sum_{d=1}^{m-1-c} \left(\frac{1}{4}\right)^{d}
        &= \frac{\frac{1}{4} - (\frac{1}{4})^{m-c}}{1 - \frac{1}{4}}\\
        &= \frac{1}{3}\left(1 - \left(\frac{1}{4}\right)^{m-1-c}\right)\\
        &= \frac{1}{3} - \frac{1}{3}\left(\frac{1}{4}\right)^{m-1-c}.
    \end{align}
    Therefore,
    \begin{align}
        \sum_{c=1}^{m-2} \left(\frac{1}{2}\right)^{2c} \sum_{d=1}^{m-1-c} \left(\frac{1}{2}\right)^{2d}
        &= \sum_{c=1}^{m-2} \left(\frac{1}{4}\right)^{c}\left(\frac{1}{3} - \frac{1}{3}\left(\frac{1}{4}\right)^{m-1-c}\right)\\
        &= \frac{1}{3}\sum_{c=1}^{m-2} \left(\frac{1}{4}\right)^{c} - \frac{1}{3}\sum_{c=1}^{m-2}\left(\frac{1}{4}\right)^{m-1}\\
        &= \frac{1}{3}\cdot \frac{\frac{1}{4} - \left(\frac{1}{4}\right)^{m-1}}{1 - \frac{1}{4}} - \frac{m-2}{3}\left(\frac{1}{4}\right)^{m-1}\\
        &= \frac{1}{3}\cdot \frac{1}{3}\left(1 - \left(\frac{1}{4}\right)^{m-2}\right) - \frac{m-2}{3}\left(\frac{1}{4}\right)^{m-1}\\
        &= \frac{1}{9} - \frac{1}{9}\left(\frac{1}{4}\right)^{m-2} - \frac{m-2}{3}\left(\frac{1}{4}\right)^{m-1}\\
        &= \frac{1}{9} - \left(\frac{1}{4}\right)^{m}\left(\frac{4m}{3} - \frac{8}{9}\right).
    \end{align}
\end{proof}

We now prove \Cref{thm1} for $n$ odd. Consider when $n=2m-1$ is odd. We again compute the likelihood of regular preferences $\pi$ such that $s_\pi(i)$ and $t_\pi(i)$ are both even. To do this, we first consider the number of $\pi$ such that $s_\pi(i)=2c$ and $t_\pi(i)=2d.$ Now, there exists such $\pi$ if and only if $c,d\ge 1$ and $c+d\le m.$ Taking such $c$ and $d$ such that $c+d\le m-1$, the probability that a randomly drawn $\pi\in \Pi_n$ satisfies $s_\pi(i)=2c$ and $t_\pi(i)=2d$ is
\begin{equation}
    \left(\frac{1}{2}\right)^{2c} \left(\frac{1}{2}\right)^{2d},
\end{equation}
for the same reason as before. However, when $c+d=m$ (corresponding to $m-1$ possible choices of $c$ and $d$), the probability is $\left(\frac{1}{2}\right)^{2m-2}$; this separate calculation is needed since the ``paths'' to the first $R$ and first $L$ overlap. So the total probability that $\pi$ is regular and $i$ is unmatched is
\begin{align}
    \sum_{\substack{c,d\ge 1 \\ c+d\le m}} \left(\frac{1}{2}\right)^{2c} \left(\frac{1}{2}\right)^{2d} &= \sum_{c=1}^{m-2} \left(\frac{1}{2}\right)^{2c} \sum_{d=1}^{m-1-c} \left(\frac{1}{2}\right)^{2d} + (m-1)\left(\frac{1}{2}\right)^{2m-2}\\
    &= \frac{1}{9} + \left(\frac{1}{2}\right)^n\left(\frac{2n}{3} - \frac{8}{9}\right),
\end{align}
as desired, where we may again use (\Cref{prop:thm1-even}).

When $\pi$ is irregular, the probability that $i$ is unmatched is $\frac{1}{n}.$ So the probability that $\pi$ is irregular and $i$ is unmatched is
\begin{equation}
    \left(\frac{1}{2}\right)^{2m-2}\left(\frac{1}{n}\right).
\end{equation}
So in total, the probability $f(n)$ that $i$ is unmatched when $n=2m-1$ is
\begin{equation}
    \frac{1}{9} + \left(\frac{1}{2}\right)^n\left(\frac{2n}{3} - \frac{8}{9} + \frac{2}{n}\right),
\end{equation}
as desired.

\end{document}